\documentclass[arxiv]{default}
\usepackage{mathdots}

\title{Joint distribution of inverses in matrix groups over finite fields}
\def\titleHead{Joint distribution of inverses in matrix groups}
\author{Corentin Perret-Gentil}

\address{Centre de Recherches Mathématiques, Université de Montréal, Canada}
\curraddr{Zürich, Switzerland}
\email{corentin.perretgentil@gmail.com}

\newcommand{\meas}{\operatorname{meas}}

\date{October 2019}

\begin{document}

\begin{abstract}
  We study the joint distribution of the solutions to the equation $gh=x$ in $G(\F_p)$ as $p\to\infty$, for any fixed $x\in G(\Z)$, where $G=\GL_n$, $\SL_n$, $\Sp_{2n}$ or $\SO_{n}^\pm$. In the special linear case, this answers in particular a question raised by S. Hu and Y. Li, and improves their error terms. Similar results are derived in certain subgroups, and when the entries of $g,h$ lie in fixed intervals. The latter shows for example the existence of $g\in\GL_n(\F_p)$ such that $g,g^{-1}$ have all entries in $[0, c_np^{1-1/(2n^2+2)+\varepsilon}]$ for some absolute constant $c_n>0$. The key for these results is to use Deligne's extension of the Weil conjectures on a sheaf on $G$, along with the stratification theorem of Fouvry, Katz and Laumon, instead of reducing to bounds on classical Kloosterman sums.
\end{abstract}

\maketitle

\section{Introduction}

Throughout, we let $n\ge 1$ be an integer, unless specified otherwise.

\subsection{The cases of $(\Z/n)^\times$ and $\GL_n(\F_p)$}

Following several similar results for the group $(\Z/n)^\times$ (see \cite{Shpar12} for a survey), Su Hu and Yan Li \cite{HuLi13} have shown that for the matrix group $G=\GL_n(\F_p)$ and any fixed $x\in G$, the solutions to the equation
\[gh=x \hspace{0.5cm} (g,h\in G)\]
are uniformly distributed in $[0,1]^{n^2}\times[0,1]^{n^2}$ as $p\to\infty$, with respect to the embedding
\begin{eqnarray}
  &&\eta: M_n(\F_p)\to [0,1]^{n^2}\label{eq:eta}\\
  &&g=(g_{i,j})_{i,j}\mapsto \left(\{g_{i,j}/p\}\right)_{i,j,}\nonumber
\end{eqnarray}
where $\{\cdot\}$ denotes the fractional part. In particular, the entries of a nonsingular matrix and its inverse are jointly uniformly distributed. More precisely, they obtain a bound for the discrepancy.

Their main tools are bounds for matrix analogues of Kloosterman sums obtained in \cite{FHLOS10} by reducing to classical Kloosterman sums.

\subsection{Special linear groups}\label{subsec:SL}

At the end of their paper, Hu and Li note that this does \textit{not} hold for $G=\SL_2(\F_p)$, but conjecture that there should be joint uniform distribution whenever $n\ge 3$. We positively answer this by showing:
\begin{theorem}\label{thm:unifdisppGLSL}
  Let $G$ be
  \begin{equation}
    \label{eq:GLSL}
    \GL_n \quad (\text{for }n\ge 2)\qquad \text{or}\qquad\SL_n \quad (\text{for }n\ge 3),
  \end{equation}
  and let $x\in G(\Z)$. As $p\to\infty$, the elements
  \[A_x(g)=\left(g, \ g^{-1}x\right)\in M_n(\F_p)\times M_n(\F_p) \hspace{0.5cm} (g\in G(\F_p))\]
  are uniformly distributed in $\Omega=[0,1]^{n^2}\times[0,1]^{n^2}$ with respect to the embedding $\eta$ in \eqref{eq:eta}. More precisely, for every product of intervals $R$ in $\Omega$,
  \[\frac{|\{g\in G(\F_p) : \eta(A_x(g))\in R\}|}{|G(\F_p)|}=\meas(R)+
    \begin{cases}
      O_{n} \left(\frac{(\log{p})^{n^2+1}}{\sqrt{p}}\right) &: G=\GL_n\\
      O_{n} \left(\frac{(\log{p})^{n^2+2}}{\sqrt{p}}\right) &: G=\SL_n
    \end{cases}
\]
as $p\to\infty$, where $\meas$ denotes the Lebesgue measure. The implied constants depend only on $n$. This also holds with $R\subset \Omega$ an arbitrary convex set if the errors are replaced by their $1/(2n^2)$th powers.
\end{theorem}

\begin{remark}
  This improves the error terms of \cite{HuLi13}, which are for example $p^{-1/(2(2n^2+1))}$ when $G=\GL_n$. The bulk of the improvement comes from bounding nontrivially the $1/r(\bs h)$ factors appearing in the Erd\H{o}s--Tur\'an--Koksma, which had been overlooked, as suggested by an anonymous referee.
\end{remark}

\begin{notation}
We recall that for two complex-valued functions $f,g$, we write $f=O_n(g)$ or $f\ll_n g$ if there exists a constant $C_n>0$, depending only on the variable $n$, such that $|f|\le C_n g$.  
\end{notation}

\subsubsection{Generalization to certain subgroups}
The following variant shows that equidistribution of $A_x(g)$ still holds in certain subgroups of $\GL_n$.

\begin{theorem}\label{thm:GLdet}
  Let us consider the setting of Theorem \ref{thm:unifdisppGLSL} for $G=\GL_n$. For $f\in \F_p[G]^\times$ a nonvanishing nonconstant function and $U\le\F_p^\times$ a subgroup, let
  \[H=f^{-1}(U)=\{g\in G(\F_p) : f(g)\in U\}.\]
  For every product of intervals $R\subset\Omega$, we have
  \[\frac{|\{g\in H : \eta(A_x(g))\in R\}|}{|H|}=\meas(R)+O_{n} \left(\frac{\sqrt{p}}{|U|}\left(\log \frac{|U|}{\sqrt{p}}\right)^{n^2+1}\right)\]
  as $p\to\infty$. The set $R$ can be replaced by an arbitrary convex set if the error term is replaced by its $1/(2n^2)$th power.
\end{theorem}
\begin{example}
  One may take $H=\{g\in \GL_n(\F_p) : \det(g)\in \F_p^{\times r}\}$ with $r=o(\sqrt{p})$, where $\F_p^{\times r}$ denotes the set of $r$-powers in $\F_p^\times$.
\end{example}
\begin{remarks}\label{ref:nonvanishingf}
  \begin{enumerate}
  \item It is a theorem of Rosenlicht (see e.g. \cite{Brou83}) that if $G$ is a connected affine algebraic group, then $f/f(1)\in\overline\F_p[G]^\times$ must be a one-dimensional character (i.e. a character of the abelianization); in particular, the set $H$ in Theorem \ref{thm:GLdet} is a \emph{normal subgroup}.
  \item In particular, we cannot get a nontrivial version of Theorem \ref{thm:GLdet} for $\SL_n(\F_p)$: since the latter is perfect for $p>3$, $f$ must be constant. The classification of maximal subgroups of $\SL_n(\F_p)$ \cite{Asch84} also shows that the restriction on the index is too stringent.
  \item Using the same techniques, it should be possible to obtain Theorem \ref{thm:GLdet} also when $H\le \GL_n(\F_p)$ is any normal subgroup of index $<\sqrt{p}$. However, this requires additional technicalities that we do not wish to pursue here (see Remark \ref{rem:FKchi} for further comments).
  \end{enumerate}
\end{remarks}

\subsection{Other classical groups}

\subsubsection{Symplectic groups}

On the other hand, it is clear that Theorem \ref{thm:unifdisppGLSL} does \textit{not} hold for $G=\Sp_n$ ($n\ge 2$ even). Indeed, if
\[g=
  \left(\begin{matrix}
    g_1&g_2\\
    g_3&g_4
  \end{matrix}\right)\in\Sp_{2n}(\F_p),\text{ then }g^{-1}=
  \left(\begin{matrix}
    g_4^t&-g_2^t\\
    -g_3^t&g_1^t
  \end{matrix}\right)\]
 (with respect to the standard symplectic form, where $g_i\in M_{2n}(\F_p)$). Hence, the obstruction for $\SL_2$ can be viewed as coming from the fact that $\SL_2(\F_p)=\Sp_2(\F_p)$.

\subsubsection{Special orthogonal groups}

Let $\Phi\in\GL_n(\F_p)$ be in one of the two equivalence classes of nonsingular symmetric bilinear forms on $\F_p^n$. Since $g^{-1}=\Phi g^t\Phi^{-1}$ for $g\in\GO(\Phi)$, Theorem \ref{thm:unifdisppGLSL} does not hold either in this case. Actually, when $n=2$, the elements of the special orthogonal group corresponding to the form $\diag(\alpha,1)$ ($\alpha\in\F_p^\times$) are themselves not uniformly distributed in $[0,1]^4$ with respect to the embedding \eqref{eq:eta}, since they are of the form $\left(
  \begin{smallmatrix}
    a&-\alpha c\\c&a
  \end{smallmatrix}
\right)$.

\subsection{Distribution of elements}

Nonetheless, the elements themselves are still uniformly distributed in all cases except $\SO_2^\pm$, as in \cite[Theorems 1.5--1.6]{HuLi13} for $\GL_n$ and $\SL_n$.

\begin{theorem}\label{thm:distrElements}
  For $n\ge 1$, let $G$ be\footnote{In what follows, we let $\SO_{n, I_n}$ be the special orthogonal group corresponding to the form given by the identity matrix $I_n$: in other words, $\SO_{n,I_n}(\F_p)$ is the special orthogonal group with square determinant, i.e. $\SO_{n}(\F_p)$ if $n$ is odd, and if $n$ is even, $\SO_{n}^\pm(\F_p)$ if $p\equiv\pm 1\pmod{4}$ respectively.}
  \[\GL_n, \quad \SL_n,\quad\Sp_{n}\ (n\text{ even}),\quad\text{ or}\quad\SO_{n,I_n} \ (n\ge 3).\]
  As $p\to\infty$, the elements $g\in G(\F_p)$ are uniformly distributed in $\Omega=[0,1]^{n^2}$ with respect to the embedding \eqref{eq:eta}. More precisely, for every product of intervals $R$ in $\Omega$,
  \[\frac{|\{g\in G(\F_p) : \eta(g)\in R\}|}{|G(\F_p)|}=\meas(R)+ O_n \left( \frac{(\log{p})^{n^2-\dim G+1}}{\sqrt{p}}\right)\]
  as $p\to\infty$. This also holds with $R\subset \Omega$ an arbitrary convex set if error term is replaced by its $1/(2n^2)$th power.
\end{theorem}

\begin{remark}
  Note that the exponent of the logarithms in the error term is $2$, $3$, $(n^2+1)/2$ if $G=\GL_n$, $\SL_n$ or $\Sp_n$ respectively. Theorem \ref{thm:distrElements} improves the errors terms:
\begin{itemize}
\item in \cite{HuLi13}, handling $\GL_n$ and $\SL_n$ using \cite{HuLi12}, which are $p^{-n/(n^2+1)}$.
\item in Theorem \ref{thm:unifdisppGLSL} for the joint distribution.
\end{itemize}
\end{remark}

In the same vein as Theorem \ref{thm:GLdet}, we get the following generalization:
\begin{theorem}\label{thm:distrElementsf}
  Under the assumptions of Theorem \ref{thm:distrElements}, let $H$ be as in Theorem \ref{thm:GLdet}. Then, for any product of intervals $R\subset\Omega$,
  \[\frac{|\{g\in H : \eta(g)\in R\}|}{|H|}=\meas(R)+ O_{n} \left(\frac{\sqrt{p}}{|U|}\left(\log \frac{|U|}{\sqrt{p}}\right)^{n^2-\dim G+1}\right).\]
\end{theorem}

\subsection{Distribution with entries in intervals}

A related question in $G=(\Z/n)^\times$ is the distribution of the solutions to $gh=x$, for some fixed $x\in G$, when $1\le g,h\le p-1$ lie in fixed intervals. It is a conjecture (see \cite[Section 3.1]{Shpar12}) that for any $\varepsilon>0$ and $p$ large enough, there exist integers $g,h$ such that $gh=1\pmod{p}$ with $|g|,|h|\le p^{1/2+\varepsilon}$. The best current result seems to be $|g|,|h|\ll p^{3/4}$, due to Garaev (note the absence of a logarithmic factor).\\

In matrix groups, we can similarly fix the entries of the matrices in intervals, yielding the following:
\begin{theorem}\label{thm:intervals}
  Let $G$ be as in Theorem \ref{thm:unifdisppGLSL}. For $p$ a prime, let $E,F\subset [0,p-1]^{n^2}$ be products of intervals. Then, for any $x\in G(\Z)$, viewing $M_n(\F_p)$ embedded in $[0,p-1]^{n^2}$, the density
  \begin{equation}
    \label{eq:densityIntervals}
    \frac{|\{g\in G(\F_p): g\in E\text{ and }g^{-1}x\in F\}|}{|G(\F_p)|}
  \end{equation}
  is given by
  \[\frac{\meas(E\times F)}{p^{2n^2}} + O_{n} \left(\frac{(\log{p})^{2n^2}}{p^{\dim G/2}}\left(1+\left(\frac{\sum_{1\le k,l\le n}\meas(E_{kl})}{\sqrt{p}}\right)^{\dim G-1}\right)\right),\]
  if $E=\prod_{1\le k,l\le n}E_{kl}$.
\end{theorem}
\begin{corollary}\label{cor:intervals}
  Let $G$ be as in Theorem \ref{thm:unifdisppGLSL} and let $p$ be a prime. For any $\varepsilon>0$ and $x\in G(\Z)$, there exist $g,h\in G(\F_p)$ such that $gh=x$ and whose entries, seen in $[0,p-1]$, are all
  \[\ll_{n,\varepsilon}
    \begin{cases}
      p^{1-\frac{1}{2(n^2+1)}+\varepsilon}&:G=\GL_n\\
      p^{1-\frac{1}{2(n^2+2)}+\varepsilon}&:G=\SL_n.
    \end{cases}
  \]
\end{corollary}
We also refer the reader to \cite{AhmSpar07} for related questions concerning matrices, and to \cite{Fouv00}, \cite[Corollary 1.5]{FK01} for general results about points on varieties in hypercubes.

\subsection{Higher-dimensional variant}
Using the same techniques, we can get an analogue of Theorem \ref{thm:unifdisppGLSL} for the uniform distribution of solutions to
\[g_1\dots g_r=x \hspace{0.5cm} (g_i\in G(\F_p))\]
for any $r\ge 2$ and $x$ fixed:

\begin{theorem}\label{thm:unifdisppGLSLr}
  Let $G$ and $x$ be as in Theorem \ref{thm:unifdisppGLSL}, and let $r\ge 2$ be an integer. As $p\to\infty$, the elements
  \[A_x(\bs g)=\left(g_1,\dots,g_{r-1},(g_1\dots g_{r-1})^{-1}x\right)\in M_n(\F_p)^r \hspace{0.5cm} (\bs g\in G(\F_p)^{r-1})\]
  are uniformly distributed in $\Omega=[0,1]^{rn^2}$ with respect to the embedding \eqref{eq:eta}. More precisely, for every product of intervals $R$ in $\Omega$,
  \[\frac{|\{\bs g\in G(\F_p)^{r-1} : \eta(A_x(\bs g))\in R\}|}{|G(\F_p)|^{r-1}}=
    \meas(R)+
    \begin{cases}
      O_{n,r} \left(\frac{(\log{p})^{n^2+1}}{\sqrt{p}}\right)&:G=\GL_n\\
      O_{n,r} \left(\frac{(\log{p})^{n^2+2}}{\sqrt{p}}\right)&:G=\SL_n
    \end{cases}
    \]
  as $p\to\infty$. This also holds with $R\subset \Omega$ an arbitrary convex set if the error terms are replaced by their $1/(2n^2)$th powers.
\end{theorem}

For the sake of clarity, we focus on proving the two-dimensional versions, and indicate the changes necessary for Theorem \ref{thm:unifdisppGLSLr} at the end.

\begin{acknowledgements}
  The author thanks Lucile Devin, his colleagues in Montréal, and anonymous referees for useful feedback on this work, as well as Yan Li for suggesting a modification of the proof of Proposition \ref{prop:constantGLSL} that makes it shorter and able to handle characteristic $2$. This work was partially supported by Koukoulopoulos' Discovery Grant 435272-2013 of the Natural Sciences and Engineering Research Council of Canada, and by Radziwiłł's NSERC DG grant and the CRC program.
\end{acknowledgements}
\section{Tools}

\subsection{Equidistribution and discrepancy}

For the following results, we refer the reader to \cite[Chapter 1]{DT97}. Throughout, we let $\Omega=[0,1]^k$ for some integer $k\ge 1$.

\begin{definition}\label{def:discrep}
  The discrepancy of a sequence $(\bs x_n)_{n\ge 1}$ in $\Omega$ is
  \begin{eqnarray*}
    D_N(\bs x_n)&=&\sup_{I\subset\Omega} \left|\frac{|\{n\le N : \bs x_n\in I\}|}{N}-\meas(I)\right|,
  \end{eqnarray*}
  where $I$ runs over all products of intervals in $\Omega$.
\end{definition}

\begin{proposition}\label{prop:discrepancy}
  A sequence $(\bs x_n)_{n\ge 1}$ in $\Omega$
is uniformly distributed if and only if $D_N(\bs x_n)=o(1)$, and we have the \emph{Erd\H{o}s--Tur\'an--Koksma inequality}: for any integer $T\ge 1$
\[D_N(\bs x_n)\le \left(\frac{3}{2}\right)^k \left(\frac{2}{T+1}+\sum_{\substack{\bs h\in\Z^k\\0<||\bs h||_\infty\le T}}\frac{1}{r(\bs h)}\left|\frac{1}{N}\sum_{n\le N}e(\bs h\cdot \bs x_n)\right|\right),\]
where $r(\bs h)=\prod_{i=1}^k\max(1,|h_i|)$, $e(z)=\exp(2\pi iz)$.
\end{proposition}
\begin{proof}
  See \cite[Theorem 1.6, Theorem 1.21]{DT97} respectively.
\end{proof}
\begin{remark}\label{rem:isotropic}
  If the sets $I$ in Definition \ref{def:discrep} are replaced by arbitrary convex subsets, this yields the isotropic discrepancy $J_N(\bs x_n)$, which satisfies $J_N(\bs x_n)\ll_k D_N(\bs x_n)^{1/k}$ (see \cite[Theorem 1.12]{DT97}).
\end{remark}

By Weyl's criterion, $(\bs x_n)_{n\ge 1}$ is equidistributed in $\Omega$ if and only if
\[\sum_{n\le N} e(\bs h\cdot \bs x_n)=o(N)\]
for every nonzero $\bs h\in\Z^k$, so the Erd\H{o}s--Tur\'an--Koksma inequality quantifies the equidistribution from the rate of decay of these exponential sums.

\subsection{Exponential sums on matrix groups}

The bounds of \cite{FHLOS10} used in \cite{HuLi13} proceed by reducing to classical Kloosterman sums on $\F_p$, through averaging and interchanging summations. Instead, we use Deligne's extension of his proof of the Weil conjectures \cite{Del2} to work directly with the sums over the matrix groups. This allows a precise control of when the sums exhibit cancellation.

\begin{proposition}\label{prop:expSumG}
  Let $G$ be as in Theorem \ref{thm:distrElements}, let $f\in\F_p(G)$ be a rational function on $G$, let $\psi: \F_p\to\C^\times$ be a nontrivial character, let $\chi:\F_p^\times\to\C^\times$ be a multiplicative character, and let $f_1\in\F_p[G]^\times$ be a nonvanishing nonconstant function. Then
  \begin{equation}
    \label{eq:expSumG}
    \frac{1}{|G(\F_p)|}\sum_{\substack{g\in G(\F_p)\\ f(g)\neq\infty}} \psi(f(g))\chi(f_1(g))=\delta + O\left(p^{-1/2}\right)
  \end{equation}
  with $\delta=1$ if $f$ and $\chi\circ f_1$ are constant on $\{g\in G(\F_p) : f(g)\neq\infty\}$, $\delta=0$ otherwise. The implied constant depends only on $n$, $\deg(f)$ and $\deg(f_1)$.
\end{proposition}
\begin{proof}
  The result is obvious if $f$ and $\chi\circ f_1$ are constant on $G(\F_p)$, so we may assume it is not the case and prove \eqref{eq:expSumG} with $\delta=0$. Let $\ell\neq p$ be an auxiliary prime. Following \cite[Exposé 6]{DelEC}, let $\Lc_0:=f^*\Lc_{\psi}=\Lc_{\psi(f)}$ (resp. $\Lc_1:=f_1^*\Lc_\chi=\Lc_{\chi(f_1)}$) be the restriction to $G$ of the Artin--Schreier (resp. Kummer) sheaf on $\A^{n^2}_{\F_p}$ corresponding to $\psi\circ f$ (resp. $\chi\circ f_1$), and let $\Lc=\Lc_0\otimes\Lc_1$ be the middle tensor product. These can be seen as representations
  \[\rho_0, \ \rho_1, \ \rho=\rho_0\otimes\rho_1: \Gal(\F_p(G)^\sep/\F_p(G))\to\overline\Q_\ell^\times,\]
  such that at every point $g\in G(\F_p)\subset \F_p^{n^2}$ with $f_1(g)\neq \infty$, there is a Frobenius element $\Frob_g$ with
  \[\iota\rho_0(\Frob_g)=\psi(f(g)), \ \iota\rho_1(\Frob_g)=\chi(f_1(g)), \ \iota\rho(\Frob_g)=\psi(f(g))\chi(f_1(g))\]
  for an embedding $\iota:\overline\Q_\ell\to\C$. Hence, the left-hand side of \eqref{eq:expSumG} is
  \[\frac{1}{|G(\F_p)|}\sum_{\substack {g\in G(\F_p) \\ f(g)\neq \infty}} \iota\rho(\Frob_g).\]
  
  By the Grothendieck--Lefschetz trace formula \cite[Exposé 6, (1.1.1)]{DelEC}, this is
  \[\frac{1}{|G(\F_p)|}\sum_{i=0}^{2\dim G}(-1)^i\iota\tr(\Frob_p \mid H^{i}_c(U\times\overline{\F}_p, \Lc)),\]
  for $U$ the open in $\A^{n^2}_{\F_p}$ where $\Lc_0$ is lisse (i.e. the complement of the zero set of $f$).

  By Deligne's extension of the Riemann hypothesis over finite fields \cite[Théorème 2]{Del2} (see also \cite[Théorème 1.17]{DelEC}), the eigenvalues of the Frobenius acting on $H^{i}_c(U\times\overline{\F}_p, \Lc)$ are $p$-Weil numbers of weight at most $i$. If the one-dimensional sheaf $\Lc$ is not geometrically trivial, the coinvariant formula implies that $H^{2\dim G}_c(U\times\overline{\F}_p, \Lc)=0$, so that the left-hand side of \eqref{eq:expSumG}
  \begin{eqnarray}
    \label{eq:sumBetti}
    \ll p^{-1/2}\sum_{i=0}^{2\dim G-1}\dim H^i_c(U\times\overline{\F}_p,\Lc).
  \end{eqnarray}
  By \cite[Theorem 12]{KatzBetti01}, the sum of Betti numbers in the error term is bounded by a quantity depending only on $n$, $\deg(f)$ and $\deg(f_1)$, for example
  \begin{equation}
    \label{eq:sumBetti2}
    3 \left(2+\max(\deg(f), n+2)+\deg(f_1)\right)^{3n^2}.
  \end{equation}
  
  Thus, it suffices to show that $\Lc$ is not geometrically trivial to conclude. If it is not the case, since $\Lc_1$ is tame everywhere and $\Lc_0$ is not unless it is geometrically trivial, we then have that both $\Lc_1$ and $\Lc_0$ are geometrically trivial. Since $\pi_1(U,\overline\eta)/\pi_1(\overline U,\overline\eta)\cong\Gal(\overline\F_q/\F_q)$, it follows as in \cite[Proposition 8.5]{FKM15} that $\psi\circ f$ and $\chi\circ f_1$ are constant on $U(\F_p)$. The former implies that $f$ is of the form $f_2^p-f_2+c$ for some $c\in\F_p^\times$ and $f_2\in\F_p(G)$, whence $f$ is constant on $U(\F_p)$ as well.
\end{proof}
\begin{remark}
  The function $f=\det^p-\det$ is not a counterexample to the theorem since, while not constant on $\GL_n(\overline\F_p)$, it is constant on $\GL_n(\F_p)$. Alternatively, note that the implied constant in \eqref{eq:expSumG} depends on $\deg(f)=p$. In the following, we will always consider cases where $\deg(f),\deg(f_1)$ are independent from $p$. Similarly, $f_1=\det^{\ord\chi}$ is not a counterexample since $\chi\circ f_1$ is constant on $G(\F_p)$.
\end{remark}

\subsubsection{Improved error terms via stratification}

The anonymous referee of Hu and Li's paper indicated (see \cite[Section 4]{HuLi13}) that the stratification results of Laumon, Katz and Fouvry may be employed to answer the conjecture for $\SL_n$ ($n\ge 3$; see Section \ref{subsec:SL}). This is not necessary to obtain uniform distribution (Proposition \ref{prop:expSumG} suffices), but we can indeed use the powerful results of Fouvry--Katz \cite{FK01} to improve the error terms.

\begin{definition}
  We consider the inner product on $M_n(\F_p)$ given by
\[g_1\cdot g_2:=\tr(g_1^tg_2)=\sum_{1\le i,j\le n}(g_1)_{i,j}(g_2)_{i,j} \hspace{0.5cm} (g_1,g_2\in M_n(\F_p)).\]  
\end{definition}

The following provides a better bound on average over shifts. We will see in Section \ref{sec:proofs} that these types of sums precisely arise when bounding discrepancies of the sequences we consider.
\begin{proposition}\label{prop:stratification}
  Under the hypotheses and notations of Proposition \ref{prop:expSumG}, assume that $f$ is obtained by reduction of a morphism of $\Z$-schemes $\hat f: G\to\A^1_\Z$. If $\delta=0$, then, for every integer $2\le T<p$,
  \[\sum_{\substack{h\in M_n(\Z)\\||h||_\infty\le T}}\frac{1}{r(h)}\left|\frac{1}{|G(\F_p)|}\sum_{g\in G(\F_p)} \psi(f(g)+h\cdot g)\chi(f_1(g))\right|\ll \frac{(\log{T})^{n^2-\dim G+1}}{p^{1/2}},\]
  where the implied constant depends only on $n$ and $\deg(\hat f)$.
\end{proposition}
\begin{proof}
  
  By \cite[Theorem 1.1, Section 3]{FK01}, there exist closed subschemes $X_j\subset\A^{n^2}_\Z$ ($0\le j\le n^2$) of relative dimension $\le n^2-j$, depending on $G_\Z$ and $\hat f$, such that
  \[X_{n^2}\subset X_{n^2-1}\subset\dots\subset X_1\subset X_0:=\A^{n^2}_\Z\]
  and
  \begin{equation}
    \label{eq:FK}
    \frac{1}{|G(\F_p)|}\sum_{g\in G(\F_p)} \psi(f(g)+h\cdot g)\chi(f_1(g))\ll \frac{p^{\frac{j-1}{2}}}{|G(\F_p)|^{1/2}}
  \end{equation}
  if $h\in M_n(\F_p)\backslash X_j(\F_p)$, identifying $M_n(\F_p)$ with $\F_p^{n^2}$ (in the case of $G=\GL_n$, the ambient space is $\A^{n^2+1}_\Z$, with an additional coordinate for $1/\det$, and one replaces the $X_j$, $0\le j\le n^2+1$, given by ibid. with their projections to the first $n^2$ coordinates).

  According to the second-to-last line of the proof of \cite[Theorem 3.1]{FK01} (from which Theorem 1.1 in ibid. follows), the implied constant in \eqref{eq:FK} is bounded by the sum of Betti numbers appearing in \eqref{eq:sumBetti} above, bounded by \eqref{eq:sumBetti2}, which only depends on $n$ and on the degree of $\hat f$ (alternatively, one may also see \cite[(3.1.2), (3.4.2)]{KatzLa85}, which controls the dependency on $\hat f$ of the implied constant in \cite[Theorem 3.1, Theorem 2.1]{FK01}).

  If $\delta=0$, we get by Proposition \ref{prop:expSumG} and \eqref{eq:FK} that
    \begin{eqnarray}
      \label{eq:FK2}
      &&\sum_{\substack{h\in\ M_n(\Z)\\||h||_\infty\le T}}\frac{1}{r(h)}\left|\frac{1}{|G(\F_p)|}\sum_{g\in G(\F_p)} \psi(f(g)+h\cdot g)\chi(f_1(g))\right|\\
      &\ll&\sum_{j=0}^{d-1}\frac{p^{\frac{j}{2}}}{|G(\F_p)|^{1/2}}\sum_{\substack{h\in M_n(\Z)\\||h||_\infty\le T}} \frac{\delta_{h\in X_j(\F_p)\backslash X_{j+1}(\F_p)}}{r(h)}+ \frac{1}{p^{1/2}}\sum_{\substack{h\in M_n(\Z)\\||h||_\infty\le T}} \frac{\delta_{h\in X_d(\F_p)}}{r(h)},\nonumber
    \end{eqnarray}
    where $d=\dim G$. By induction as in \cite[Lemma 9.5]{FK01} and \cite[Lemma 1.7]{Xu18}, we get that
    \[\sum_{\substack{h\in M_n(\Z)\\ ||h||_\infty\le T}}\frac{\delta_{h \, (\text{mod }p)\in X_j(\F_p)}}{r(h)}\ll (\log{T})^{\dim X_j}.\]
    Therefore, \eqref{eq:FK2} is
    \begin{eqnarray*}     
      &\ll&\sum_{j=0}^{d-1}\frac{p^{\frac{j}{2}}}{|G(\F_p)|^{1/2}}(\log{T})^{n^2-j}+ \frac{1}{p^{1/2}}(\log{T})^{n^2-d}\\
      &=&(\log{T})^{n^2}\left(\frac{1}{|G(\F_p)|^{1/2}}\sum_{j=0}^{d-1}\left(\frac{\sqrt{p}}{\log{T}}\right)^{j}+ \frac{1}{p^{1/2}(\log{T})^d}\right),
    \end{eqnarray*}
    where the implied constants depend only on $n$ and $\deg(\hat f)$.
\end{proof}
\begin{remark}\label{rem:FKchi}
  To handle normal subgroups $H\le G(\F_p)$ as suggested in Remark \ref{ref:nonvanishingf}, we would need to replace $\chi\circ f_1$ by a character $\chi$ of $G(\F_p)$ (or of $G(\F_p)/H$). To do so, one would consider the Lang torsor $\Lc_1$ corresponding to $\chi$ as in \cite[1.22-25]{DelEC}. Since all centralizers in $\GL_n$ are connected, ibidem shows that the trace function associated to $\Lc_1$ yields the character $\chi$. One could then proceed as in the proofs of \cite[Corollary 3.2, Theorem 1.1]{FK01}.
\end{remark}
\begin{remark}
  Under the non-vanishing of an ``$A$-number'', \cite[Theorem 1.2]{FK01} shows that the exponent in \eqref{eq:FK} can be improved to $\max(0,j/2-1)$, giving a nontrivial bound whenever $j<d+2$. This would be nontrivial for all $j$ with $G=\SL_n$ as well. However, we cannot use \cite[Theorem 8.1]{FK01} to show the non-vanishing, since $|G(\F_p)|\equiv 0\pmod{p}$ (see \cite[Chapter 3]{Wil09}).
\end{remark}

\section{Proofs of Theorems \ref{thm:unifdisppGLSL}, \ref{thm:GLdet}, \ref{thm:distrElements} and \ref{thm:distrElementsf}}
\label{sec:proofs}
\subsection{Setup of the exponential sums}
To obtain the theorems from Proposition \ref{prop:discrepancy}, we need to bound sums of the form
\begin{equation}
  \label{eq:sums}
  \frac{1}{|H|}\sum_{g\in H}\psi\left(h_1\cdot g+h_2 \cdot (g^{-1}x)\right)
\end{equation}
for $H\normal G(\F_p)$, $x\in G(\Z)$ and $h_1,h_2\in M_n(\F_p)$, where $\psi(x)=e(x/p)$. Note that in Theorems \ref{thm:unifdisppGLSL} and \ref{thm:distrElements}, we simply have $H=G(\F_p)$.

By the orthogonality relations, \eqref{eq:sums} can be written as 

\begin{eqnarray}
  &&\frac{1}{|G(\F_p)/H|}\sum_{\chi\in\widehat{G(\F_p)/H}} \overline\chi(g)\left(\frac{1}{|H|}\sum_{g\in G(\F_p)} \psi\left(h_1\cdot g+h_2 \cdot f(g)\right) \chi(g)\right)\nonumber\\
  &\ll&\sum_{\chi\in\widehat{G(\F_p)/H}}\left|\frac{1}{|G(\F_p)|}\sum_{g\in G(\F_p)} \psi\left(h_1\cdot g+h_2 \cdot f(g)\right) \chi(g)\right|\label{eq:sums2},
\end{eqnarray}
with $f(g)=g^{-1}x$.

Under the assumptions of the theorems, $G(\F_p)/H$ is either trivial or isomorphic to a quotient $\F_p^\times/U$ for a subgroup $U\le\F_p^\times$ (since $H=\ker(G\xrightarrow{f_1}\F_p^\times\to\F_p^\times/U)$), setting $U=\F_p^\times$ if $H=G(\F_p)$.  Hence, \eqref{eq:sums2} is
\begin{equation}
  \label{eq:sums3}
  \sum_{\substack{\chi\in\widehat{\F_p^\times}\\\chi\mid_U=1}}\left|\frac{1}{|G(\F_p)|}\sum_{g\in G(\F_p)} \psi\left(h_1\cdot g+h_2 \cdot f(g)\right) \chi(f_1(g))\right|.
\end{equation}

By Proposition \ref{prop:expSumG}, the inner sum is small whenever the rational function on $G$ appearing in $\psi$ is nonconstant. We determine when this is the case in the next section.

\subsection{Constant functions on $G$}

\subsubsection{The case of $\GL_n$ ($n\ge 2$) and $\SL_n$ ($n\ge 3$)}

\begin{proposition}\label{prop:constantGLSL}
  Let $G$ be as in \eqref{eq:GLSL}, let $x\in G(\F_p)$, and let $h_1,h_2\in M_n(\F_p)$. We assume that $p\ge 3$ if $n=2$. If
  \[\big(g\in G(\F_p)\big)\mapsto h_1\cdot g+h_2 \cdot (g^{-1}x)\]
  is constant, then $h_1=h_2=0$.
\end{proposition}
\begin{proof}
  Since $h_2\cdot (g^{-1}x)=(h_2x^t)\cdot g^{-1}$, it suffices to prove the result when $x=1$.

    With the identity matrix and the elementary matrices $g=I+e_{i,j}\in\SL_n(\F_p)$ for $1\le i, j\le n$ distinct, we get that $(h_1)_{i,j}=(h_2)_{i,j}$.

  When $G=\GL_2$ and $p\neq 2$, the matrices $g=\left(
    \begin{smallmatrix}
      \lambda&0\\0&1
    \end{smallmatrix}
  \right)\in\GL_2(\F_p)$ with $\lambda\in\F_p^\times$ show that
  \[(\lambda\in\F_p^\times)\mapsto \lambda (h_1)_{1,1}+\lambda^{-1}(h_2)_{1,1},\]
  is constant, so that $(h_1)_{1,1}=(h_2)_{1,1}=0$ and the diagonals of $h_1$ and $h_2$ are zero by symmetry. Similarly, the matrices $g=\lambda\left(
    \begin{smallmatrix}
      0&1\\\pm 1&0
    \end{smallmatrix}
  \right)$ with $\lambda\in\F_p^\times$ show that $(h_1)_{1,2}=\mp (h_1)_{2,1}$ and $(h_2)_{1,2}=\pm (h_2)_{2,1}$, whence $h_1=h_2=0$. Thus, we may now suppose that $n\ge 3$.

  For $1\le i,j,k\le n$ distinct, the matrix $g=I-e_{i,j}-e_{j,k}\in\SL_n(\F_p)$, with inverse $g^{-1}=I+e_{i,j}+e_{j,k}+e_{i,k}$, gives
  \[\tr(h_1+h_2)=h_1\cdot g+h_2\cdot g^{-1}=\tr(h_1+h_2)+(h_2)_{i,k},\]
  so that $(h_2)_{i,k}=0$ and $h_2$ is diagonal. By symmetry, the same holds for $h_1$.

  Using the matrices
  \[g=\left(
     \begin{matrix}
       &&&&(-1)^{n+1}\\
       1&&&&\lambda\\
       &1&&&0\\
       &&\ddots&&\vdots\\
       &&&1&0
    \end{matrix}
  \right), \ g^{-1}=\left(
     \begin{matrix}
       (-1)^n\lambda&1&&&\\
       0&&1&&\\
       \vdots&&&\ddots&\\
       0&&&&1\\
       (-1)^{n+1}&&&&\\
    \end{matrix}
  \right)\]
in $\SL_n(\F_p)$, for $\lambda\in\F_p$, we get that $(\lambda\in\F_p)\mapsto (-1)^n\lambda(h_2)_{1,1}$ is constant, so that $(h_2)_{1,1}=0$. By symmetry, $(h_1)_{1,1}=0$ as well.

Finally, if $x\in\GL_n(\F_p)$, we note that
\[h_1\cdot(x^{-1}gx)+h_2\cdot(x^{-1}g^{-1}x)=(x^{-t}h_1x^t)\cdot g+(x^{-t}h_2x^t)\cdot g^{-1},\]
which shows that we may permute the diagonal elements of $h_1$ and $h_2$. By the previous steps, $h_1=h_2=0$.
\end{proof}

\begin{remark}
 By the affine linear sieve of Bourgain, Gamburd and Sarnak \cite{BGS10} and the work of Salehi-Golsefidy--Varj\'u and others, this implies the following: Let $n\ge 3$, $S$ be a finite symmetric generating set for $\SL_n(\Z)$ and $(\gamma_N)_{N\ge 0}$ be a random walk on the Cayley graph of $\SL_n(\Z)$ with respect to $S$, starting at $1$, i.e.
 \[\gamma_{N+1}=\xi_{N+1}\gamma_N\text{ for }N\ge 0,\text{ with }\xi_{N+1}\text{ uniform in }S.\]
 Then, for any $h_1,h_2\in M_{n}(\Z)$ that are not both zero, there exists $M\ge 1$ such that
 \[P \Big(h_1\cdot \gamma_N+h_2\cdot \gamma_N^{-1}\text{ has }\le M\text{ prime factors}\Big)\asymp 1/N\]
 as $N\to+\infty$.
\end{remark}

\subsubsection{Symplectic groups}

\begin{proposition}
  Let $n\ge 2$, $G=\Sp_{2n}$, $x\in G(\F_p)$, and $h_1,h_2\in M_{2n}(\F_p)$. Then
  \[\big(g\in G(\F_p)\big)\mapsto h_1\cdot g+h_2 \cdot g^{-1}x\]
  is constant if and only if
  \begin{equation}
    \label{eq:SpConstant}
    h_1=\left(
      \begin{matrix}
        h_{11}&h_{12}\\
        h_{13}&h_{14}
      \end{matrix}
    \right), \ h_2=\left(
      \begin{matrix}
        -h_{14}^t&h_{12}^t\\
        h_{13}^t&-h_{11}^t
      \end{matrix}
\right)x^{-t},
\end{equation}
where $h_{1i}\in M_{n}(\F_p)$ $(1\le i\le 4)$.
\end{proposition}
\begin{proof}
  Again, it suffices to consider the case $x=1$. With respect to the standard form,
  \[g=\left(
      \begin{matrix}
        A&0\\0&A^{-t}
      \end{matrix}
    \right), \ \left(
      \begin{matrix}
        0&A\\-A^t&0
      \end{matrix}
    \right)\in\Sp_{2n}(\F_p)\]
  for any $A\in\GL_n(\F_p)$. The result then follows from applying Proposition \ref{prop:constantGLSL}.
\end{proof}

\subsubsection{Special orthogonal groups}

\begin{proposition}
  For $n\ge 3$, let $G=\SO_{n, I_n}$. Then, for $p\ge 3$ and $h\in M_{n}(\F_p)$,
  \[\big(g\in G(\F_p)\big)\mapsto h\cdot g\]
  is constant if and only if $h=0$.
\end{proposition}
\begin{proof}
  Any permutation matrix $g_\sigma\in\SL_n(\F_p)$ with $\sigma\in A_n$ belongs to\footnote{If $G$ corresponded instead to the form $\diag(\alpha,1,\dots,1)$, $\alpha\neq 1$, this would be true only for the permutations fixing $1$.} $G(\F_p)$. If $\sigma=(i\,j\,k)\in A_n$ is a cycle of length $3$, the matrices $g_{\sigma}(I-2e_j-2e_k)$ and $g_\sigma$ show that $h$ is diagonal. For $1\le i,j\le n$ distinct, the matrices $I-2e_i-2e_j$ show that the diagonal of $h$ is zero.
\end{proof}

\subsection{Proof of Theorems \ref{thm:unifdisppGLSL} and \ref{thm:GLdet}}
By Proposition \ref{prop:discrepancy}, for any integer $1\le T<p$, the discrepancy $D_N(A_x(g))$ is
\begin{eqnarray*}
  &\ll&\frac{1}{T}+\sum_{\substack{\bs h\in M_n(\Z)^2\\0<||h||_\infty\le T}}\frac{1}{r(\bs h)}\left|\frac{1}{|H|}\sum_{g\in H}\psi\left(h_1\cdot g+h_2\cdot f(g)\right)\right|.
\end{eqnarray*}
By \eqref{eq:sums3}, the second summand is
\begin{eqnarray*}
  &&\ll (\log{T})^{n^2}\max_{\substack{h_2\in M_n(\Z)\\||h_2||_\infty\le T}}\sum_{\substack{h_1\in M_n(\Z)\\||(h_1,h_2)||_\infty\le T\\ (h_1,h_2)\neq 0}}\frac{1}{r(h_1)}\\
  &&\hspace{3cm}\sum_{\substack{\chi\in\widehat{\F_p^\times}\\\chi|_U=1}}\left|\frac{1}{|G(\F_p)|}\sum_{g\in G(\F_p)} \psi\left(h_1\cdot g+h_2 \cdot f(g)\right) \chi(f_1(g))\right|.
\end{eqnarray*}
By Proposition \ref{prop:stratification} and Proposition \ref{prop:constantGLSL}, we get
\begin{eqnarray*}
  D_N(A_x(g))&\ll&\frac{1}{T}+\frac{|G(\F_p)|}{|H|}\frac{(\log{T})^{2n^2-\dim G+1}}{\sqrt{p}}\\
             &\ll&\frac{|G(\F_p)/H|}{\sqrt{p}}\log\left(\frac{\sqrt{p}}{|G(\F_p)/H|}\right)^{2n^2-\dim G+1},
\end{eqnarray*}
taking $T=\floor{\sqrt{p}/|G(\F_p)/H|}$.

The last statements in Theorems \ref{thm:unifdisppGLSL} and \ref{thm:GLdet} follow from Remark \ref{rem:isotropic}.

\begin{remark}
  Using Proposition \ref{prop:expSumG} only instead of Proposition \ref{prop:stratification} would have given an exponent of the logarithm equal to $2n^2$.
\end{remark}
\subsection{Proof of Theorems \ref{thm:distrElements} and \ref{thm:distrElementsf}}
Similarly, by Propositions \ref{prop:discrepancy}, \ref{prop:stratification} and \eqref{eq:sums3}, for $1\le T<p$, the discrepancy $D_N(\eta(g))$ is
\begin{eqnarray*}
  &\ll&\frac{1}{T}+\sum_{\substack{h\in\Z^{n^2}\\0<||h||_\infty\le T}} \frac{1}{r(h)}\left|\sum_{g\in H}\frac{1}{|H|}\psi(h\cdot g)\right|\\
              &\ll&\frac{1}{T}+\sum_{\chi\in\widehat{G(\F_p)/H}}\sum_{\substack{h\in M_n(\Z)\\0<||h||_\infty\le T}}\frac{1}{r(h)}\left|\frac{1}{|G(\F_p)|}\sum_{g\in G(\F_p)} \psi\left(h\cdot g\right) \chi(g)\right|\\
              &\ll&\frac{1}{T}+\frac{|G(\F_p)|}{|H|}\frac{(\log{T})^{n^2-\dim G+1}}{p^{1/2}}\ll\frac{|G(\F_p)/H|}{\sqrt{p}}\log\left(\frac{\sqrt{p}}{|G(\F_p)/H|}\right)^{n^2-\dim G+1}
\end{eqnarray*}

\begin{remark}
  As above, using Proposition \ref{prop:expSumG} instead of Proposition \ref{prop:stratification} would have given an exponent of the logarithm equal to $n^2$. Note that these exponents in the case of $\GL_n$ or $\SL_n$ do not depend on $n$.
\end{remark}

\subsection{Higher-dimensional variant} To obtain Theorem \ref{thm:unifdisppGLSLr}, we need to control sums of the form
\begin{equation}
  \label{eq:sumr}
  \sum_{\bs g\in G^{r-1}(\F_p)}\psi \left(\sum_{i=1}^{r-1}h_i\cdot g_i+h_{r}(g_1\dots g_{r-1})^{-1}x\right)
\end{equation}
for $\bs h=(h_1,\dots,h_r)\in M_n(\F_p)^{r}$. To do so, it suffices to replace $G$ by $G^{r-1}$ and $M_n(\F_p)^2$ by $M_n(\F_p)^{r}$ in the arguments above. In the first bound, there is no dependency with $r$ in the exponent since we average over all but one $h_i$, and we can use Proposition \ref{prop:stratification}. From Proposition \ref{prop:constantGLSL}, we see that the rational function in \eqref{eq:sumr} is constant if and only if $\bs h=0$.

\section{Proof of Theorem \ref{thm:intervals} and Corollary \ref{cor:intervals}}

\subsection{Proof of Theorem \ref{thm:intervals}}

By orthogonality, we can write the density \eqref{eq:densityIntervals} as 
\begin{eqnarray}
  &&\frac{1}{|G(\F_p)|}\sum_{g\in G(\F_p)} \sum_{\substack{e\in E_p\\f\in F_p}} \frac{1}{p^{2n^2}}\sum_{u,v\in M_n(\F_p)}\psi \left(u\cdot (g-e)+v\cdot(g^{-1}x-f)\right)\nonumber\\
  &=&\frac{1}{p^{2n^2}}\sum_{u,v\in M_n(\F_p)}\sum_{e\in E_p}\overline\psi(u\cdot e)\sum_{f\in F_p}\overline\psi(v\cdot f)S(u,v)\label{eq:densityError}\\
  &=&\frac{|E_p||F_p|}{p^{2n^2}}+ O \left(\frac{1}{p^{2n^2}}\sum_{\substack{u,v\in M_n(\F_p)\\(u,v)\neq 0}}\left|\sum_{e\in E_p}\overline\psi(u\cdot e)\right|\left|\sum_{f\in F_p}\overline\psi(v\cdot f)\right||S(u,v)|\right),\nonumber
\end{eqnarray}
where $E_p=E\pmod{p}$, $F_p=F\pmod{p}\subset\F_p$ and 
\[S(u,v)=\frac{1}{|G(\F_p)|}\sum_{g\in G(\F_p)}\psi \left(u\cdot g+v\cdot(g^{-1}x)\right).\]

Since $E=\prod_{1\le k,l\le n}E_{kl}$ is a product of intervals, Weyl's bound gives
\begin{eqnarray*}
  \sum_{e\in E_p}\overline\psi(u\cdot e)&=&\prod_{1\le k,l\le n}\sum_{e_{kl}\in E_{kl}}\overline\psi(u_{kl} e_{kl})\\
                                      &\ll&\prod_{1\le k,l\le n}\min\left(|E_{kl}|,||u_{kl}/p||^{-1}\right),
\end{eqnarray*}
and similarly for $F$, with $||\cdot||$ denoting the distance to the nearest integer and $|E_{kl}|:=\meas(E_{kl})$. Hence, the error term in \eqref{eq:densityError} is
\begin{eqnarray*}
  &\ll&\frac{1}{p^{n^2}}\sum_{v\in M_n(\F_p)}\prod_{1\le k,l\le n}\min\left(|F_{kl}|,||v_{kl}/p||^{-1}\right)\\
  &&\times\frac{1}{p^{n^2}}\sum_{\substack{u\in M_n(\F_p)\\(u,v)\neq 0}}|S(u,v)|\prod_{1\le k,l\le n}\min\left(|E_{kl}|,||u_{kl}/p||^{-1}\right).
\end{eqnarray*}
To bound the sum over $u$, we proceed as in Proposition \ref{prop:stratification}, using \cite{FK01}. With $d=\dim G$,
\begin{eqnarray}
  &&\frac{1}{p^{n^2}}\sum_{\substack{u\in M_n(\F_p)\\(u,v)\neq 0}}|S(u,v)|\prod_{1\le k,l\le n}\min\left(|E_{kl}|,||u_{kl}/p||^{-1}\right)\label{eq:sumu}\\
  &\ll&\left(\frac{1}{p^{d/2}}\sum_{j=0}^{d-1} p^{j/2}+\sum_{j=d}^{n^2} p^{-1/2} \right)\frac{1}{p^{n^2}}\sum_{u\in X_j(\F_p)}\prod_{1\le k,l\le n}\min\left(|E_{kl}|,||u_{kl}/p||^{-1}\right).\nonumber
\end{eqnarray}
By \cite[Lemma 9.5]{FK01} (or \cite[(2.6)]{Fouv00}), if $X\subset \A^{n^2}$ has dimension $\le n^2-j$,
\begin{equation*}
  \label{eq:FK95}
  \sum_{u\in X(\F_p)}\prod_{1\le k,l\le n}\min\left(|E_{kl}|,||u_{kl}/p||^{-1}\right)\ll (p\log{p})^{n^2-j}M_E^{j},
\end{equation*}
where $M_E=\max_{k,l}|E_{kl}|$. Proceeding by induction as in op. cit., we get the more precise bound
\begin{equation}
  \label{eq:FK952}
  \sum_{u\in X(\F_p)}\prod_{1\le k,l\le n}\min\left(|E_{kl}|,||u_{kl}/p||^{-1}\right)\ll (p\log{p})^{n^2-j} e_{j}(|E_{kl}|)
\end{equation}
when the $|E_{kl}|$ may not be all equal, where $e_{j}$ is the $j$th elementary symmetric polynomial in $n^2$ variables.

Thus, \eqref{eq:sumu} is
\begin{eqnarray*}
  &\ll&\left(\frac{1}{p^{d/2}}\sum_{j=0}^{d-1} p^{j/2}+\sum_{j=d}^{n^2} p^{-1/2} \right) (\log{p})^{n^2}e_j(|E_{kl}|) p^{-j}\\
  &\ll&(\log{p})^{n^2}\left(\frac{1}{p^{d/2}}\sum_{j=0}^{d-1} e_j\left(\frac{|E_{kl}|}{\sqrt{p}}\right) +\frac{1}{\sqrt{p}}\sum_{j=d}^{n^2} e_j\left(\frac{|E_{kl}|}{p}\right) \right)\\
  &\ll&(\log{p})^{n^2}\left(\frac{1}{p^{d/2}}\sum_{j=0}^{d-1} e_j\left(\frac{|E_{kl}|}{\sqrt{p}}\right) +\frac{1}{\sqrt{p}}e_d\left(\frac{|E_{kl}|}{p}\right) \right).
\end{eqnarray*}

By Maclaurin's inequality \cite[(12.3)]{Stee04}, letting $L_E=e_1(|E_{kl}|)$, this is
\begin{eqnarray*}
  &\ll&(\log{p})^{n^2}\left(\frac{1}{p^{d/2}}\sum_{j=0}^{d-1} \left(L_E/\sqrt{p}\right)^{j}+\frac{(L_E/p)^d}{\sqrt{p}} \right)\\
  &\ll&(\log{p})^{n^2}\left(\frac{1}{p^{d/2}}\max \left(1, \left(L_E/\sqrt{p}\right)^{d-1}\right)+\frac{(L_E/p)^d}{\sqrt{p}} \right)\\
  &\ll&\frac{(\log{p})^{n^2}}{p^{d/2}}\max \left(1, \left(L_E/\sqrt{p}\right)^{d-1}\right).
\end{eqnarray*}

Using \eqref{eq:FK952} again, we find that the total error in \eqref{eq:densityError} is
\begin{eqnarray*}
  &\ll&\frac{(\log{p})^{2n^2}}{p^{d/2}}\max \left(1, \left(L_E/\sqrt{p}\right)^{d-1}\right)
\end{eqnarray*}
\subsubsection{Proof of Corollary \ref{cor:intervals}}

Finally, if $E$ and $F$ are the products of intervals of the same integral length $x$, then the density \eqref{eq:densityIntervals} is
\[\left(\frac{x}{p}\right)^{2n^2}+ O_n \left(\frac{(\log{p})^{2n^2}}{p^{d/2}}\max \left(1, \left(\frac{x}{\sqrt{p}}\right)^{d-1}\right)\right).\]
The main term dominates if and only if
\[x\gg_{n,\varepsilon} p^{1-\frac{1}{2(2n^2-\dim G+1)}+\varepsilon}\]
for any $\varepsilon>0$, which yields the corollary.

\bibliographystyle{alpha}
\bibliography{references}

\newcommand{\etalchar}[1]{$^{#1}$}
\begin{thebibliography}{FHL{\etalchar{+}}10}

\bibitem[AS07]{AhmSpar07}
Omran Ahmadi and Igor~E. Shparlinski.
\newblock Distribution of matrices with restricted entries over finite fields.
\newblock {\em Indag. Math. (N.S.)}, 18(3):327--337, 2007.

\bibitem[Asc84]{Asch84}
Michael Aschbacher.
\newblock On the maximal subgroups of the finite classical groups.
\newblock {\em Inventiones mathematicae}, 76(3):469--514, 1984.

\bibitem[BGS10]{BGS10}
Jean Bourgain, Alex Gamburd, and Peter Sarnak.
\newblock Affine linear sieve, expanders, and sum-product.
\newblock {\em Inventiones mathematicae}, 179(3):559--644, March 2010.

\bibitem[Bro83]{Brou83}
Sean~Allen Broughton.
\newblock A note on characters of algebraic groups.
\newblock {\em Proceedings of the American Mathematical Society}, 89(1):39--40,
  1983.

\bibitem[Del77]{DelEC}
Pierre Deligne.
\newblock {\em Cohomologie étale, séminaire de géométrie algébrique du
  {B}ois-{M}arie {SGA} 4$\frac{1}{2}$}, volume 569 of {\em Lecture notes in
  {M}athematics}.
\newblock Springer, 1977.

\bibitem[Del80]{Del2}
Pierre Deligne.
\newblock La conjecture de {W}eil. {II}.
\newblock {\em Publications Mathématiques de l'Institut des Hautes Études
  Scientifiques}, 52(1):137--252, 1980.

\bibitem[DT97]{DT97}
Michael Drmota and Robert~F. Tichy.
\newblock {\em Sequences, discrepancies, and applications}.
\newblock Number 1651 in Lecture notes in mathematics. Springer, 1997.

\bibitem[FHL{\etalchar{+}}10]{FHLOS10}
Ron Ferguson, Corneliu Hoffman, Florian Luca, Alina Ostafe, and Igor~E.
  Shparlinski.
\newblock Some additive combinatorics problems in matrix rings.
\newblock {\em Revista Matemática Complutense}, 23(2):501--513, July 2010.

\bibitem[FK01]{FK01}
Etienne Fouvry and Nicholas~M. Katz.
\newblock A general stratification theorem for exponential sums, and
  applications.
\newblock {\em Journal für die reine und angewandte Mathematik},
  2001(504):115--166, 2001.

\bibitem[FKM15]{FKM15}
\'{E}tienne Fouvry, Emmanuel Kowalski, and Philippe Michel.
\newblock Algebraic twists of modular forms and {H}ecke orbits.
\newblock {\em Geom. Funct. Anal.}, 25(2):580--657, 2015.

\bibitem[Fou00]{Fouv00}
Étienne Fouvry.
\newblock Consequences of a result of {N}. {Katz} and {G}. {Laumon} concerning
  trigonometric sums.
\newblock {\em Israel Journal of Mathematics}, 120(1):81--96, 2000.

\bibitem[HL12]{HuLi12}
Su~Hu and Yan Li.
\newblock Gauss sums over some matrix groups.
\newblock {\em Journal of Number Theory}, 132(12):2967 -- 2976, dec 2012.

\bibitem[HL13]{HuLi13}
Su~Hu and Yan Li.
\newblock On a uniformly distributed phenomenon in matrix groups.
\newblock {\em Journal of Number Theory}, 133(11):3578--3588, November 2013.

\bibitem[Kat01]{KatzBetti01}
Nicholas~M. Katz.
\newblock Sums of {Betti} numbers in arbitrary characteristic.
\newblock {\em Finite fields and their Applications}, 7(1):29--44, 2001.

\bibitem[KL85]{KatzLa85}
Nicholas~M. Katz and G\'{e}rard Laumon.
\newblock Transformation de {F}ourier et majoration de sommes exponentielles.
\newblock {\em Inst. Hautes \'{E}tudes Sci. Publ. Math.}, (62):361--418, 1985.

\bibitem[Shp12]{Shpar12}
Igor~E. Shparlinski.
\newblock Modular hyperbolas.
\newblock {\em Japanese Journal of Mathematics}, 7(2):235--294, November 2012.

\bibitem[Ste04]{Stee04}
J.~Michael Steele.
\newblock {\em The {C}auchy-{S}chwarz master class}.
\newblock MAA Problem Books Series. Mathematical Association of America,
  Washington, DC; Cambridge University Press, Cambridge, 2004.
\newblock An introduction to the art of mathematical inequalities.

\bibitem[Wil09]{Wil09}
Robert~A. Wilson.
\newblock {\em The {Finite} {Simple} {Groups}}, volume 251 of {\em Graduate
  {Texts} in {Mathematics}}.
\newblock Springer, 2009.

\bibitem[Xu18]{Xu18}
Junyan Xu.
\newblock {Stratification for Multiplicative Character Sums}.
\newblock {\em International Mathematics Research Notices}, May 2018.
\newblock rny096.

\end{thebibliography}

\end{document}